\documentclass{amsart}

\usepackage[all]{xy}
\usepackage{amsmath,amssymb, amsthm}
\usepackage{enumerate}
\usepackage{hyperref}
  \usepackage{graphicx}
  \usepackage{latexsym} 
  \usepackage{amsfonts} 
  \usepackage[2emode]{psfrag}
\usepackage{tikz, pgf}

\newtheorem{proposition}{Proposition}[section] 
\newtheorem{lemma}[proposition]{Lemma}
\newtheorem{corollary}[proposition]{Corollary}
\newtheorem{theorem}[proposition]{Theorem}

\theoremstyle{definition}
\newtheorem{Definition}[proposition]{Definition}
\newtheorem{example}[proposition]{Example}

\theoremstyle{remark}

\newcommand{\thlabel}[1]{\label{th:#1}}
\newcommand{\thref}[1]{Theorem~\ref{th:#1}}

\newcommand{\lelabel}[1]{\label{le:#1}}
\newcommand{\leref}[1]{Lemma~\ref{le:#1}}

\newcommand{\colabel}[1]{\label{co:#1}}
\newcommand{\coref}[1]{Corollary~\ref{co:#1}}

\newcommand{\exlabel}[1]{\label{ex:#1}}
\newcommand{\exref}[1]{Example~\ref{ex:#1}}

\newcommand{\eqlabel}[1]{\label{eq:#1}}
\newcommand{\equref}[1]{(\ref{eq:#1})}

\newcommand{\Hom}{{\sf Hom}}

\newcommand{\End}{{\sf End}}

\newcommand{\Aut}{{\sf Aut}\,}

\newcommand{\Char}{{\sf Char}\,}
\newcommand{\Mod}{{\sf Mod}}

\def\ot{\otimes}

\def\lie{{[-,-]}}
\def\dual{{<-,->}}

\def\MM{{\mathbb M}}

\def\ZZ{{\mathbb Z}}

\newcommand{\Cc}{\mathcal{C}}
\newcommand{\Dd}{\mathcal{D}}

\newcommand{\Hh}{\mathcal{H}}

\newcommand{\Mm}{\mathcal{M}}

\def\text#1{{\rm {\rm #1}}}

\def\lim{{\rm lim\,}}

\def\H{{\sf H}}

\def\Vect{{\sf Vect}}
\def\En{{\sf En}}

\title{A note on the categorification of Lie algebras}
\author{Isar Goyvaerts}
\address{Department of Mathematics, Vrije Universiteit Brussel, Pleinlaan 2, B-1050 Brussel, Belgium} \email{igoyvaer@vub.ac.be}
\author{Joost Vercruysse}
\address{D\'epartement de Math\'ematiques, Universit\'e Libre de Bruxelles, Boulevard du Triomphe, B-1050 Bruxelles, Belgium}
\email{jvercruy@ulb.ac.be}

\begin{document}

\begin{abstract}
In this short note we study Lie algebras in the framework of symmetric monoidal categories. After a brief review of the existing work in this field and a presentation of earlier studied and new examples, we examine which functors preserve the structure of a Lie algebra.
\end{abstract}

\maketitle

\section{Introduction}
Lie algebras have many generalizations such as Lie superalgebras, Lie color and $(G,\chi)$-Lie algebras, braided Lie algebras, Hom-Lie algebras, Lie algebroids, etc.

Motivated by the way that the field of Hopf algebras benefited from the interaction with the field of monoidal categories (see e.g.\ \cite{Tak}) on one hand, and the strong relationship between Hopf algebras and Lie algebras on the other hand, the natural question arose whether it is possible to study Lie algebras within the framework of monoidal categories, and whether Lie theory could also benefit from this viewpoint.

First of all, it became folklore knowledge that Lie algebras can be easily defined in any symmetric monoidal $k$-linear category over a commutative ring $k$, or (almost equivalently) in any symmetric monoidal additive category. Within this setting, many (but not all) of the above cited examples can already be recovered. We will treat slightly in more detail the examples of Lie superalgebras and Hom-Lie algebras in the second section.

As some examples, in particular Lie color algebras, do not fit into this theory, several attemps were made to define Lie algebras in any {\em braided}, rather than {\em symmetric} monoidal category. A reason to do this is that $G$-graded modules over any group $G$ give rise to a monoidal category, whose center is a braided monoidal category that can be described as the category of Yetter-Drinfel'd modules over a Hopf algebra. In this way, Lie color algebras and $(G,\chi)$-Lie algebras are recovered as a special case (see \cite{P2}). A slightly different point of view is advocated by Majid, whose motivation is to describe deformations of Lie algebras, that he calls braided Lie algebras, inside a braided monoidal category, such that the universal enveloping of this deformed Lie algebra encodes the same information as the deformed (quantum) enveloping algebra of the original Lie algebra (see \cite{Majid}).

We will not discuss further these two last cited types of Lie algebras in this short note. Rather, we will study Lie algebras in a (possibly non-symmetric, possibly non-braided) monoidal category, such that the Lie algebra allows a local symmetry. That is, the Lie algebra possesses a self-invertible Yang-Baxter operator and the anti-symmetry and Jacobi identity are defined up to this Yang-Baxter operator.

\section{Lie algebras in additive monoidal categories}

Troughout, we will work in a symmetric monoidal and additive category. 
Without any change in the arguments, one can work in any $k$-linear symmetric monoidal category, where $k$ is a commutative ring with characteristic different from $2$.

\begin{Definition}
Let $\Cc=(C,\ot,I,a,l,r,c)$ be a symmetric monoidal additive category with associativity constraint $a$, left- and right unit constraints resp. $l$ and $r$ and symmetry $c$. A {\em Lie algebra} in $\Cc$ is a pair $(L,\lie)$, where $L$ is an object of $\Cc$ and $\lie:L\ot L\to L$ is a morphism in $\Cc$ that satisfies the following two conditions
\begin{eqnarray}
\lie \circ (id_{L\ot L} + c_{L,L})&=&0_{L\ot L,L},\\
\lie\circ (id_{L}\ot \lie)\circ (id_{L\ot(L\ot L)}+ t+ w)&=&0_{L\ot (L\ot L),L},
\end{eqnarray}
where $t=c_{L\ot L,L}\circ a^{-1}_{L,L,L}$ and $w=a_{L,L,L}\circ c_{L,L\ot L}$.
\end{Definition}

\begin{example}\exlabel{LieAlg} 
Let $\Mm_{R}=(\Mod(R),\ot_{R},R,a,l,r,c)$ be the abelian, symmetric monoidal category of (right) $R$-modules over a commutative ring $R$ ($\Char(R)\neq 2$) with trivial associativity and unit constraints and with symmetry $c=\tau$ (the flip). Taking a Lie algebra in $\Mm_{R}$, one obtains the classical definition of a Lie algebra over $R$.
\end{example}
\begin{example}\exlabel{SuperLie}
Let $\Cc=(\Vect^{\ZZ_{2}}(k),\ot_{k},k,a,l,r,c)$ be the abelian, symmetric monoidal category of $k$-vector spaces ($\Char(k)\neq 2$) graded by $\ZZ_{2}$. We take the trivial associativity and unit constraints. The symmetry $c$ is defined as follows: For any pair of objects $(V,W)$ in $\Cc$; $c_{V,W}:V\ot W\to W\ot V; v\ot w\mapsto (-1)^{|v||w|}w\ot v$. 
Taking a Lie algebra in $\Cc$, one recovers the definition of a Lie superalgebra (see also \cite{P2}).
\end{example}
\begin{example}
We now recall from \cite{CG}, the construction of a non-trivial example of an abelian, non-strict symmetric monoidal category (called the Hom-construction). \\
Let $\Cc$ be a category. A new category $\Hh(\Cc)$ is introduced as follows:
objects are couples $(M,\mu)$, with $M\in \Cc$ and $\mu\in \Aut_\Cc(M)$.
A morphism $f:\ (M,\mu)\to (N,\nu)$ is a morphism $f:\ M\to N$ in $\Cc$ such that
$\nu\circ f=f\circ \mu.$

Now assume that $\Cc=(\Cc,\ot,I,a,l,r,c)$ is a braided monoidal category. 
Then one easily verifies that $\Hh(\Cc)=(\Hh(\Cc),\ot,(I,I),a,l,r,c)$ is again a braided monoidal category, with the tensor product defined by the following formula 
\begin{equation}\nonumber
(M,\mu)\ot (N,\nu)=(M\ot N,\mu\ot \nu),
\end{equation}
for $(M,\mu)$ and $ (N,\nu)$ in $\Hh(\Cc)$.
On the level of morphisms, the tensor product is the tensor products of morphisms in $\Cc$.
By deforming the category $\Hh(\Cc)$, we obtain the category $\widetilde{\Hh}(\Cc)=(\Hh(\Cc),\ot,(I,I),\tilde{a},\tilde{l},\tilde{r},c)$ which is still a braided monoidal category (but no longer strict if $\Cc$ was strict).
The associativity constraint $\tilde{a}$ is given by the formula
\begin{equation}\nonumber
\tilde{a}_{M,N,P}=a_{M,N,P}\circ ((\mu\ot N)\ot \pi^{-1})
=(\mu\ot (N\ot \pi^{-1}))\circ a_{M,N,P},
\end{equation}
for $(M,\mu),(N,\nu),(P,\pi)\in \Hh(\Cc)$. The unit constraints $\tilde{l}$ and
$\tilde{r}$ are given by
\begin{equation}\nonumber
\tilde{l}_M=\mu\circ l_M=l_M\circ (I\ot \mu)~~;~~
\tilde{r}_M= \mu\circ r_M=r_M\circ (\mu\ot I).
\end{equation}
Now, A Lie algebra in $\widetilde{\Hh}(\Mm_{R})$ is a triple $(L,[-,-],\alpha)$ with $(L,\alpha)\in\widetilde{\Hh}(\Mm_{R})$, $[-,-]:\ L\ot L\to L$ a morphism
in $\widetilde{\Hh}(\Mm_R)$ (that is, $[\alpha(x)\ot\alpha(y)]=\alpha[x,y]$), satisfying anti-symmetry and the so-called Hom-Jacobi identity;
$$[\alpha(x)\ot [y\ot z]]+[\alpha(y)\ot [z\ot x]]+ [\alpha(z)\ot [x\ot y]]=0,$$
We thus recover the definition of a Hom-Lie algebra (cf.\cite{HS}), where in this case $\alpha$ is a classical Lie algebra isomorphism. 
\end{example}
\begin{example}
\item A Lie coalgebra in $\Cc$ is a Lie algebra in $\Cc^{op}$, the opposite category of $\Cc$. This means that a Lie coalgebra is a pair $(C,\dual)$, where $\dual:C\to C\ot C$ is a map that satisfies the following two conditions
\begin{eqnarray*}
(id_{C\ot C} + c_{C,C})\circ \dual&=&0;\\
(id_{C\ot(C\ot C)}+ t+ w)\circ (id_{C}\ot \dual)\circ \dual&=&0.
\end{eqnarray*}
Lie coalgebras were introduced by Michaelis \cite{Mich}.
\end{example}

Our next aim is to `free' the definition of Lie algebra of the global symmetry on our additive monoidal category. 

\begin{Definition}
Let $\Cc=(C,\ot,I,a,l,r)$ be a (possibly non-symmetric) 
monoidal category and $L$ an object in $\Cc$. A {\em self-invertible Yang-Baxter operator} on $L$ is a morphism $c:L\ot L\to L\ot L$ that satisfies the following conditions: 
\begin{eqnarray}
c\circ c&=&L\ot L; \eqlabel{zelfinvers}\\
&&\hspace{-3cm}a_{L,L,L}\circ(c\ot L)\circ a_{L,L,L}^{-1}\circ(L\ot c)\circ a_{L,L,L}\circ(c\ot L)\eqlabel{YB} \\
&=&(L\ot c)\circ a_{L,L,L}\circ (c\ot L)\circ a_{L,L,L}^{-1}\circ(L\ot c)\circ a_{L,L,L} \nonumber
\end{eqnarray}
\end{Definition}

Given an object $L$ in $\Cc$, together with a self-invertible Yang Baxter operator $c$ as above, we can construct the following morphisms in $\Cc$:
\begin{eqnarray*}
t=t_c:=
a_{L,L,L}\circ (c\ot L)\circ a^{-1}_{L,L,L}\circ (L\ot c);\\
w=w_c:=
(L\ot c)\circ a_{L,L,L}\circ (c\ot L)\circ a^{-1}_{L,L,L}.
\end{eqnarray*}
One can easily verify that $t\circ t=w$ and $t\circ w=id=w\circ t$.

\begin{example}\exlabel{symmetric}
If $\Cc$ is a symmetric monoidal category, with symmetry $c_{X,Y}:X\ot Y\to Y\ot X$, for all $X,Y\in\Cc$, then $c_{L,L}$ is a self-invertible Yang-Baxter operator for $L\in \Cc$. Obviously, $c_{L,L}$ satisfies conditions \equref{zelfinvers}; to see that $c_{L,L}$ also satisfies \equref{YB}, one applies the hexagon condition in combination with the naturality of $c$. 
Moreover, $t_{c_{L,L}}=c_{L\ot L,L}\circ a^{-1}_{L,L,L}$ and $w_{c_{L,L}}=a_{L,L,L}\circ c_{L,L\ot L}$.
\end{example}

\begin{Definition}
Let $\Cc$ be an an additive, monoidal category, but not necessarily symmetric. A YB-Lie algebra in $\Cc$ is a triple $(L,\lambda,\lie)$, where $L$ is an object of $\Cc$, $\lambda$ is a self-invertible Yang-Baxter operator for $L$ in $\Cc$, and $\lie:\ L\ot L\to L$ is a morphism in $\Cc$ that satisfies
\begin{eqnarray}\eqlabel{AS}
\lie \circ (id_{L\ot L} + \lambda)&=&0_{L\ot L,L},\\
\eqlabel{Jac}
\lie\circ (id_{L}\ot \lie)\circ (id_{L\ot(L\ot L)}+ t_\lambda+ w_\lambda)&=&0_{L\ot (L\ot L),L}.\\
(id_L\ot \lie)\circ t_\lambda\circ a_{L,L,L}&=&\lambda\circ (\lie\ot id_L)\eqlabel{compatibility}
\end{eqnarray}
We call \equref{Jac} the (right) $\lambda$-Jacobi identity for $L$. The equation \equref{compatibility} expresses the compatibility between the Lie bracket $\lie$ and the Yang-Baxter operator $\lambda$. Remark that in the case were $\lambda=c_{L,L}$ (see \exref{symmetric}), this condition is automatically satisfied by the naturality of the symmetry $c_{-,-}$.
\end{Definition}

As for usual Lie algebras, the definition of a YB-Lie algebra is left-right symmetric, as follows from the following Lemma.

\begin{lemma}
Let $(L,\lambda,\lie)$ be a YB-Lie algebra in $\Cc$. Then $L$ also satisfies the left $\lambda$-Jacobi identity, that is the following equation holds
$$\lie\circ (\lie\ot id_{L})\circ a^{-1}_{L,L,L}\circ (id_{L\ot(L\ot L)}+ t_\lambda+ w_\lambda)=0_{L\ot (L\ot L),L}. $$
\end{lemma}
\begin{proof}
Using \equref{AS} in the first equality, \equref{compatibility} in the second equality,  $w_\lambda=t_\lambda^2=t_\lambda^{-1}$ in the third equality and \equref{Jac} in the last equality we find
\begin{eqnarray*}
&&\lie\circ (\lie\ot id_{L})\circ a^{-1}_{L,L,L}\circ (id_{L\ot(L\ot L)}+ t_\lambda+ w_\lambda)\\
&=&-\lie\circ\lambda\circ (\lie\ot id_{L})\circ a^{-1}_{L,L,L}\circ(id_{L\ot(L\ot L)}+ t_\lambda+ w_\lambda)\\
&=&-\lie\circ (id_{L}\ot \lie)\circ t_\lambda\circ a_{L,L,L}\circ a^{-1}_{L,L,L}\circ (id_{L\ot(L\ot L)}+ t_\lambda+ w_\lambda)\\
&=&-\lie\circ (id_{L}\ot \lie)\circ (id_{L\ot(L\ot L)}+ t_\lambda+ w_\lambda)=0
\end{eqnarray*}
\end{proof}
\begin{example}
\exlabel{monad}
Let $\Cc$ be any additive category, and consider the functor category $\End(\Cc)$ of endofunctors on $\Cc$ and natural transformations between them. Recall that this is a monoidal category with the composition of functors as tensor product on objects and the Godement product as tensor product on morphisms. 
Moreover, $\End(\Cc)$ inherits the additivity of $\Cc$.
We will call a YB-Lie algebra in $\End(\Cc)$ a {\em Lie monad} on $\Cc$. 
\end{example}

\begin{example}
Let $(B,\mu_B)$ be an associative algebra in an additive, monoidal category $\Cc$ and suppose there is a self-invertible Yang-Baxter operator $\lambda:B\ot B\to B\ot B$ on $B$, such that the conditions hold:
\begin{eqnarray}\eqlabel{assalgebra}
(B\ot \mu_B)\circ a^{-1}_{B,B,B}\circ w_\lambda &=& \lambda\circ (\mu_B\ot B)\\
(\mu_B \ot B)\circ t_\lambda\circ a_{B,B,B} &=& \lambda \circ (B\ot \mu_B)\nonumber
\end{eqnarray}
Then we can consider a YB-Lie algebra structure on $B$, induced by the commutator bracket $\lie_{B}$ (defined by $\lie_{B}=\mu_{B}\circ(B\ot B-\lambda)$). E.g. If $B$ is a braided Hopf algebra (or a braided bialgebra) in the sense of Takeuchi (see \cite{Tak}) then $B$ admits a Yang-Baxter operator $\lambda$ that satisfies the diagrams \equref{assalgebra}. If $\lambda$ is self-invertible, the commutator algebra of $B$ is a YB-Lie-algebra in our sense. Moreover, the primitive elements of $B$ can be defined as the equaliser $(P(B),eq)$ in the following diagram
$$\xymatrix{P(B)\ar[rr]^{eq} && B \ar@<.5ex>[rr]^-{\Delta}  \ar@<-.5ex>[rr]_-{\eta\ot B+B\ot \eta} &&B\ot B}$$
where $\Delta:B\to B\ot B$ is the comultiplication on $B$ and $\eta:k\to B$ is the unit of $B$.
One can show (see forthcoming \cite{GV}) that $P(B)$ is again a YB-Lie algebra.
\end{example}

\section{Functorial properties}

In this section we study functors that send Lie algebras to Lie algebras. 

Let $\Cc=(C,\ot,I,a,l,r)$ and $\Dd=(D,\odot,J,a',l',r')$ be two additive, monoidal categories. For simplicity, we will suppose that $\Cc$ and $\Dd$ are strict monoidal, that is $a,l,r$ and $a',l',r'$ are identity natural transformations and will be omitted. By Mac Lane's coherence theorem, this puts no restrictions on the subsequent results.

\begin{Definition}
A functor $F:\Cc\to \Dd$ will be called a {\em non-unital monoidal} functor, if there exists a natural transformation $\Psi_{X,Y}:FX\odot FY\to F(X\ot Y)$ that satisfies the following condition
\begin{eqnarray}
\Psi_{X\ot Y,Z}\circ (\Psi_{X,Y}\odot FZ) = \Psi_{X,Y\ot Z}\circ (FX\odot \Psi_{Y,Z}).\eqlabel{monoidality}
\end{eqnarray} 
\end{Definition}

\begin{lemma}\lelabel{lemma1}
Let $(F,\Psi):\Cc\to \Dd$ be a non-unital monoidal functor and use notation as above.
Let $\lambda:L\ot L\to L\ot L$ be a self-invertible Yang-Baxter operator on $L\in\Cc$. Suppose that there exists a morphism $\lambda':FL\ot FL\to FL\ot FL$ such that $\Psi_{L,L}\circ \lambda'=F(\lambda)\circ \Psi_{L,L}$. If $\Psi_{L\ot L,L}$, $\Psi_{L,L}\odot FL$ and $\Psi_{L,L}$ are monomorphisms (e.g. $\Psi$ is a natural monomorphism and the endofunctor $-\odot FL$ preserves monos), then $\lambda'$ is a self invertible Yang-Baxter operator on $FL$.
\end{lemma}

\begin{proof}
Using the compatibility between $\lambda$ and $\lambda'$ in the first equality, the naturality of $\Psi$ in the second equality, \equref{monoidality} in the third equality, a repetition of the above arguments in the fourth equality, the Yang-Baxter identity for $\lambda$ in the fifth equality, and a reverse computation in the last equality, we find
\begin{eqnarray*}
&&\Psi_{L\ot L,L}\circ(\Psi_{L,L}\odot FL)\circ(\lambda'\odot FL)\circ (FL\odot \lambda')\circ (\lambda'\odot FL)\\
&=&\Psi_{L\ot L,L}\circ (F(\lambda)\odot FL)\circ (\Psi_{L,L}\odot FL)\circ (FL\odot \lambda')\circ (\lambda'\odot FL)\\
&=&F(\lambda\ot L)\circ \Psi_{L\ot L,L}\circ (\Psi_{L,L}\odot FL)\circ (FL\odot \lambda')\circ (\lambda'\odot FL)\\
&=&F(\lambda\ot L)\circ \Psi_{L,L\ot L}\circ (FL\odot \Psi_{L,L})\circ (FL\odot \lambda')\circ (\lambda'\odot FL)\\
&=&F(\lambda\ot L)\circ F(L\ot \lambda)\circ F(\lambda\ot L)\circ \Psi_{L,L\ot L}\circ (FL\odot \Psi_{L,L})\\
&=&F(L\ot \lambda)\circ F(\lambda\ot L)\circ F(L\ot \lambda)\circ\Psi_{L,L\ot L}\circ (FL\odot \Psi_{L,L})\\
&=&\Psi_{L\ot L,L}\circ(\Psi_{L,L}\odot FL)\circ (FL\odot \lambda')\circ(\lambda'\odot FL)\circ (FL\odot \lambda')
\end{eqnarray*}
As $\Psi_{L\ot L,L}$ and $\Psi_{L,L}\odot FL$ are monomorphisms, we conclude from the computation above that $\lambda'$ satisfies the Yang-Baxter identity. In a similar way, one proofs that $\lambda'$ is self-invertible.
\end{proof}

\begin{lemma}\lelabel{lemma2}
Let $(F,\Psi):\Cc\to \Dd$ be a non-unital monoidal functor and use notation as above.
Let $\lambda:L\ot L\to L\ot L$ and  $\lambda':FL\ot FL\to FL\ot FL$ be $\Cc$ (resp.\ $\Dd$)-morphisms such that $\Psi_{L,L}\circ \lambda'=F(\lambda)\circ \Psi_{L,L}$. 
Then the following identities hold
\begin{eqnarray}
\Psi_{L\ot L,L}\circ (\Psi_{L,L}\odot FL)\circ t_{\lambda'} &=& F(t_\lambda)\circ \Psi_{L,L\ot L} \circ (FL\odot \Psi_{L,L})\eqlabel{lemma1}\\ 
\Psi_{L\ot L,L}\circ (\Psi_{L,L}\odot FL)\circ w_{\lambda'} &=& F(w_\lambda)\circ \Psi_{L,L\ot L} \circ (FL\odot \Psi_{L,L}) \eqlabel{lemma2}
\end{eqnarray}
\end{lemma}

\begin{proof}
Let us proof equation \equref{lemma1}, the proof of \equref{lemma2} is completely similar.
\begin{eqnarray*}
\Psi_{L\ot L,L}\circ (\Psi_{L,L}\odot FL)\circ t_{\lambda'} &=&
\Psi_{L\ot L,L}\circ (\Psi_{L,L}\odot FL)\circ(\lambda'\odot FL)\circ (FL\odot \lambda')\\&=&
\Psi_{L\ot L,L}\circ (F(\lambda)\odot FL)\circ(\Psi_{L,L}\odot FL)\circ (FL\odot \lambda')\\&=&
F(\lambda\ot L)\circ \Psi_{L\ot L,L}\circ(\Psi_{L,L}\odot FL)\circ (FL\odot \lambda')\\&=&
F(\lambda\ot L)\circ \Psi_{L,L\ot L}\circ(FL\odot \Psi_{L,L})\circ (FL\odot \lambda')\\&=&
F(\lambda\ot L)\circ \Psi_{L,L\ot L}\circ(FL\odot F(\lambda))\circ (FL\odot \Psi_{L,L})\\&=&
F(\lambda\ot L)\circ F(L\ot \lambda)\circ \Psi_{L,L\ot L} \circ (FL\odot \Psi_{L,L})\\&=&
F(t_\lambda)\circ \Psi_{L,L\ot L} \circ (FL\odot \Psi_{L,L})
\end{eqnarray*}
We used the compatibility between $\lambda$ and $\lambda'$ in the second and fifth equality, the naturality of $\Psi$ in the third and sixth equality and \equref{monoidality} in the fourth.
\end{proof}

Remark that the existence of the morphism $\lambda'$ as in the above lemmata is guaranteed if $F$ is a strong monoidal functor, as in this situation $\Psi$ is invertible. Furthermore, if $\Cc$ and $\Dd$ are symmetric monoidal and we take $\lambda$ and $\lambda'$ induced by the symmetry of $\Cc$ and $\Dd$ respectively, then the compatibility condition between $\lambda$ and $\lambda'$ is automatically satisfied. 

\begin{theorem}\thlabel{functorial}
Let $(F,\Psi):\Cc\to \Dd$ be an additive non-unital monoidal functor and $(L,\lambda,\lie)$ a YB-Lie algebra in $\Cc$. Suppose that there exists a self-invertible Yang-Baxter operator $\lambda':FL\ot FL\to FL\ot FL$ such that $\Psi_{L,L}\circ \lambda'=F(\lambda)\circ \Psi_{L,L}$.
Then $(FL,\lambda',\lie')$ is a YB-Lie algebra in $\Dd$ with Lie-bracket given by
$$\xymatrix{\lie':FL\odot FL \ar[rr]^-{\Psi_{L,L}} && F(L\ot L) \ar[rr]^-{F(\lie)} &&FL}.$$
\end{theorem}

\begin{proof}
Let us check that $\lie'$ is antisymmetric. Using the antisymmetry of $(L,\lie)$ and compatibility between $\lambda$ and $\lambda'$ we obtain
\begin{eqnarray*}
\lie'\circ \lambda'&=& F(\lie)\circ \Psi_{L,L}\circ \lambda'=F(\lie)\circ F(\lambda)\circ \Psi_{L,L}\\
&=&F(\lie\circ \lambda)\circ \Psi_{L,L}=-F(\lie)\circ \Psi_{L,L}=-\lie'.
\end{eqnarray*}
Next, let us check the Jacobi identity
\begin{eqnarray*}
&&\lie'\circ (id_{FL}\odot \lie')\circ (id_{FL\odot(FL\odot FL)}+ t_{\lambda'}+w_{\lambda'})\\
&=&F(\lie)\circ\Psi_{L,L}\circ(id_{FL}\odot F(\lie))\circ (id_{FL}\odot \Psi_{L,L})\circ(id+ t_{\lambda'}+w_{\lambda'})\\
&=&F(\lie)\circ F(id_L\odot \lie)\circ \Psi_{L,L\ot L}\circ (id_{FL}\odot \Psi_{L,L})\circ(id+ t_{\lambda'}+w_{\lambda'})\\
&=&F(\lie)\circ F(id_L\odot \lie)\circ F(id_{L\ot L\ot L} + t_{\lambda} + w_{\lambda})\circ \Psi_{L\ot L,L}\circ (\Psi_{L,L}\odot id_{FL})\\
&=&0
\end{eqnarray*}
We used the naturality of $\Psi$ in the second equality and \leref{lemma2} in the third equation and \equref{Jac} in the last equality.
\end{proof}

Combining \thref{functorial} with \leref{lemma1}, we immediately obtain the following two satisfying corollaries, which allow to apply \thref{functorial} in practical situations.

\begin{corollary}\colabel{symmetric}
Let $(F,\Psi_0,\Psi):\Cc\to\Dd$ be an additive symmetric monoidal functor between additive symmetric monoidal categories. If $(L,\lie)$ is a Lie algebra in $\Cc$, then $(FL,\Psi_{L,L}\circ\lie)$ is a Lie algebra in $\Dd$.
\end{corollary}

\begin{corollary}\colabel{YB}
Let $(F,\Psi):\Cc\to \Dd$ be an additive (non-unital) strong monoidal functor between additive monoidal categories. If $(L,\lambda,\lie)$ is a YB-Lie algebra in $\Cc$, then $(FL,\Psi_{L,L}^{-1}\circ F(\lambda)\circ \Psi_{L,L},F(\lie)\circ\Psi_{L,L})$ is a YB-lie algebra in $\Dd$.
\end{corollary}

\begin{example}
Let us return to the case of Lie superalgebras, which are exactly Lie algebras in $\Vect^{\ZZ_{2}}(k)$ (cf. \exref{SuperLie}). It is well-known that this category is equivalent (even isomorphic) to the category $\Mm^{k[\ZZ_2]}$ of comodules over the groupalgebra $k[\ZZ_2]$, which is in fact a Hopf-algebra. Moreover, this equivalence of categories is an additive, monoidal equivalence, and even a symmetric one, taking into account the coquasitrangular structure on $k[\ZZ_2]$. By our \coref{symmetric}, this implies that Lie algebras can be computed equivalently in $\Vect^{\ZZ_{2}}(k)$ as well as in $\Mm^{k[\ZZ_2]}$.
In fact, Lie algebras in a general monoidal category $\Mm^H$ of comodules over a coquasitriangular bialgebra $H$ have been studied in \cite{FM} and in \cite{bath}, amongst others. Such a Lie algebra is a triple $(M,\rho_{M},[-,-])$ with $M$ a $k$-vectorspace, a coaction $\rho$ on $M$ and a $k$-linear map $[-,-]:M\ot M\to M$ such that $([-,-]\ot id_{H})\circ \rho_{M\ot M}=\rho_{M}\circ[-,-]$, which satisfy the condition \equref{AS}:
\begin{equation}\nonumber
[x\ot y]=-[(y_{[0]}\ot x_{[0]})\sigma(x_{[1]}\ot y_{[1]})]
\end{equation}
and \equref{Jac}:
\begin{align}\nonumber
[x\ot[y\ot z]]+[z_{[0][0]}\ot [x_{[0]}\ot y_{[0]}]]\sigma(y_{[1]}\ot z_{[1]})\sigma(x_{[1]}\ot z_{[0][1]})\nonumber\\
[y_{[0]}\ot [z_{[0]}\ot x_{[0][0]}]]\sigma(x_{[1]}\ot y_{[1]})\sigma(x_{[0][1]}\ot z_{[1]}).\nonumber
\end{align}    
whenever $x,y,z\in M$ and where we used the Sweedler-Heynemann for comodules and $\sigma:H\ot H\to k$ is the convolution invertible bilinear map from the coquasitriangular structure on $H$.
\end{example}

\begin{example}
Let us consider again the Hom-construction. It is proven in \cite[Proposition 1.7]{CG} that the categories $\Hh(\Cc)=
(\Hh(\Cc),\ot, (I,I), a,l,r,c)$ and $\tilde{\Hh}(\Cc)=
(\Hh(\Cc),\ot, (I,I), \tilde{a},\tilde{l},\tilde{r},c)$ are isomorphic as monoidal categories. Let us briefly recall this isomorphism.\\ 
Let $F:\ \Hh(\Cc)\to \widetilde{\Hh}(\Cc)$ be the identity functor, and
$\Psi_0:\ I\to I$ the identity. We define a natural transformation, by putting for all $M,N\in \Hh(\Cc)$, 
$$\Psi_{M,N}=\mu\ot\nu:\ F(M)\ot F(N)=M\ot N\to F(M\ot N)=M\ot N.$$
Then $(F,\Psi_0,\Psi)$ is a strict monoidal functor and it is clearly an isomorphism of categories. Moreover, if $\Cc$ is an additive category, then $F$ is also an additive functor, so $F$ preserves Lie algebras by \coref{symmetric} and YB-Lie algebras by \coref{YB}.\\
Let $((L,\alpha),\lie)$ be a Lie algebra in $\Hh(\Cc)$ 
i.e. $(L,\lie)$ is a Lie algebra in $\Cc$ with a Lie algebra isomorphism $\alpha$. Then 
$(F(L,\alpha),\lie')$ is a Lie algebra in $\widetilde{\Hh}(\Cc)$. The inverse functor is also strict monoidal and additive, hence preserves Lie algebras. Consequently, Hom-Lie algebras, where $\alpha$ is a Lie algebra isomorphism, are nothing else than Lie algebras endowed with a Lie algebra isomorphism.
\end{example}

\begin{example}
Multiplier algebras serve as an important tool to study certain types of non-compact quantum groups, within the framework of multiplier Hopf algebras, see \cite{VD:Multi}. In \cite{JV1} it was proven that the creation of the multiplier algebra of a non-degenerated idempotent (non-unital) $k$-algebra leads to a (symmetric) monoidal (additive) functor $(\MM,\Psi_0,\Psi)$ on the category of these algebras. Hence the multiplier construction preserves Lie algebras by our \thref{functorial}. Moreover, as the monoidal product on the category of non-degenerated idempotent (non-unital) $k$-algebras is given by the monoidal product of underlying $k$-vectorspaces it follows that the multiplier construction also preserves the commutator Lie algebras associated to these algebras. Furthermore, the natural transformation $\Psi$ is a natural monomorphism. Therefore, we can apply \leref{lemma1} and the functor $\MM$ also preserves YB-Lie algebras.
\end{example}

\begin{example}
Let $\Cc$ be an additive monoidal category, and consider the additive monoidal category $\End(\Cc)$ from \exref{monad}.
Consider the functor $\En:\Cc\to \End(\Cc)$, that sends every object $X\in \Cc$ to the endofunctor $-\ot X:\Cc\to \Cc$. Then $\En$ is an additive strong monoidal functor. By \coref{YB}, a YB-Lie algebra in $\Cc$ leads to a YB-Lie algebra in $\End(\Cc)$, i.e.\ to a Lie monad on $\Cc$.

Suppose now that $\Cc$ is a right closed monoidal category, i.e.\ every endofunctor $-\ot X$ has a right adjoint, that we denote by $\H(X,-):\Cc\to \Cc$. Then there exist natural isomorphisms $\pi_{Y,Z}:\Hom(Y,\H(X,Z))\to \Hom(Y\ot X,Z)$. One can proof (see e.g.\ \cite{GV}) that this isomorphism can be extended to an isomorphism
$$\H(X,\H(Y,-))\cong \H(X\ot Y,-)$$
in $\Cc$. Hence the contravariant functor $\H:\Cc\to \End(\Cc)$ that sends an object $X\in \Cc$ to the endofunctor $\H(X,-)$ is a strong monoidal functor. As  consequence, this functor sends a YB-Lie coalgebra in $\Cc$ to a Lie monad on $\Cc$. This idea is further explored in \cite{GV} to study dualities between infinite dimensional Hopf algebras and Lie algebras.
\end{example}

\end{document}